\documentclass[12pt, reqno]{amsart}
\usepackage{amsmath, amsthm, amscd, amsfonts, amssymb, graphicx, color}
\usepackage[bookmarksnumbered, colorlinks, plainpages]{hyperref}

\newtheorem{theorem}{Theorem}[section]

\newtheorem{corollary}[theorem]{Corollary}
\theoremstyle{definition}

\newtheorem{example}[theorem]{Example}

\newtheorem{conjecture}[theorem]{Conjecture}

\theoremstyle{remark}
\newtheorem{remark}[theorem]{Remark}
\numberwithin{equation}{section}
\def\mc{\mathcal}

\begin{document}
\setcounter{page}{1}

\title[Fuglede-Putnam type commutativity theorems for  $ EP $ operators]{Fuglede-Putnam type commutativity theorems for  $ EP $ operators}

\author[P. Sam Johnson, Vinoth  A., K. Kamaraj]{P. Sam Johnson$^1$, Vinoth  A.$^{1,2}$$^{*}$ and  K. Kamaraj$^3$}

\address{$^{1}$ Department of Mathematical and Computational Sciences, 
	National Institute of Technology Karnataka, Surathkal, Mangaluru 575 025, India.}
\email{sam@nitk.edu.in}

\address{$^{2}$ Department of Mathematics, St. Xavier's College, Palayamkottai 627 002, India.}
\email{vinoth.antony1729@gmail.com}

\address{$^{3}$ 	Department of Mathematics, University  College of Engineering Arni, Anna University, Arni 632 326, India.}
\email{krajkj@yahoo.com}



\keywords{Fuglede-Putnam theorem, Moore-Penrose inverse, $EP$ operator.
\newline \indent $^{*}$ Corresponding author}


\begin{abstract}
Fuglede-Putnam theorem is not true in general for $ EP $ operators on Hilbert spaces. We prove that under some conditions the theorem holds good. If the adjoint operation is replaced by Moore-Penrose inverse in the theorem, we get Fuglede-Putnam type theorem for $ EP $  operators -- however proofs are totally different.  Finally, interesting results on $ EP $ operators have been proved using several versions of Fuglede-Putnam type theorems for $ EP $ operators on Hilbert spaces.
\end{abstract} \maketitle

\section{Introduction}
A square matrix $A$ over the complex field $\mathbb{C}$ is said to be an $EP$ matrix if ranges of $A$ and $A^*$ are equal. Although the  $EP$ matrix was defined by Schwerdtfeger \cite{Hans1950} in 1950,  it could not get any greater attention until Pearl \cite{Pearl1966(2)} characterized it through Moore-Penrose inverse in 1966. The normed space of all bounded linear operators from a Hilbert space $\mc H $ to a Hilbert space $\mc K $ is denoted by $ \mc B(\mc H, \mc K ) $.  We write $ \mc B(\mc H, \mc H)=\mc B(\mc H) $.  If $ T\in \mc B(\mc H, \mc K) $, we denote the kernel of $ T $ by $ \mc N(T) $ and the range of $ T $ by $ \mc R(T) $.  The operator $ T $ is said to be invertible if its inverse exists and is bounded.  Given $ T\in \mc B(\mc H, \mc K) $, $ S\in \mc B(\mc K, \mc H) $ is the adjoint operator on $\mc  H $ if $ \langle Tx, y\rangle = \langle x, Sy\rangle $ for all $ x\in \mc H $ and $ y\in \mc K $; in this case the operator $ S $ is denoted by $ T^*$. If $ T\in \mc B(\mc H, \mc K) $ with a closed range, then $T^\dag$ is the unique linear operator in $\mathcal \mc B(\mc K, \mc H)$ satisfying $$TT^\dag T=T, \ T^\dag TT^\dag =T^\dag, \ TT^\dag = (TT^\dag)^* \ \text{ and  } \ T^\dag T = (T^\dag T)^*.$$

The operator $ T^\dag $ is called the Moore-Penrose inverse of $ T $. It is well-known that an operator $ T $ has a closed range if and only if its Moore-Penrose inverse $ T^\dag $ exists. The class $ \mc B_c(\mc H)$ denotes the set of all operators in $ \mc B(\mc H) $ having closed ranges.  For any nonempty set $ \mc M $ in $ \mc H $, $ \mc M^\perp $ denotes the orthogonal complement of $ \mc M $.  Note that if $ T\in \mc B_c(\mc H) $, then $ T^*\in \mc B_c(\mc H)$, $\mc N(T)^\perp=\mc R(T^*)$, $ \mc N(T^*)^\perp = \mc R(T) $ and $ \mc R(T)=\mc R(TT^*)$.  An operator $T\in \mc B_c(\mc H)$ is said  to be an $EP$ operator if $\mc R(T)=\mc R(T^*)$. $EP$ matrices and operators have been studied by many authors \cite{Pearl1966,Meyer1970,Campbell1975,Brock1990,Hartwig1997,Koliha2007,Dijana2017,Xu2019}. It is well-known that if $T$
is normal with a closed range, or an invertible operator, then $T$ is $EP$. The converse is not true even in a finite dimensional
space.

The Fuglede-Putnam theorem (first proved by B. Fuglede \cite{Fuglede1950} and then by C. R. Putnam \cite{Putnam1951} in a more general version) plays a major role in the theory of bounded (and unbounded) operators. Many authors have worked on it since the papers of Fuglede and Putnam got published \cite{Duggal2001,Gong1987,Gupta1988,Mecheri2004}. 
There are various generalizations of the Fuglede-Putnam theorem to non-normal operators, for instance, 
hyponormal, subnormal, etc.  This paper is devoted to the study of Fuglede-Putnam type theorems for $ EP $ operators.

In section 2, we give some known characterizations for $ EP $ operators and we give a procedure to construct an $ EP $ matrix $ T $ (preferably non-normal) for the given subspace $ \mc W $ of the unitary space $\mathbb{C}^n $ such that $ \mc R(T)=\mc W $. This construction has been used in the paper to construct suitable examples of $ EP $ matrices. We show in section 3  that the Fuglede theorem \cite{Fuglede1950} is not true in general for $ EP $ operators (Example \ref{ex_1}) and we prove that the commutativity relation in Fuglede-Putnam theorem is true for $ EP $ operators if the adjoint operation is replaced by Moore-Penrose inverse.  Moreover, several versions of Fuglede-Putnam type theorems are given for $ EP $ operators.  In the last section, we prove some interesting results using Fuglede-Putnam type theorems for $EP$ operators on Hilbert spaces.

\section{Preliminaries}
Let $\mc H $ be a complex Hilbert space.  An operator on $\mc H $ means a linear operator from $\mc H $ into itself.  Given an $ EP $ operator $ T $ on $\mc H $, we get a closed subspace $\mc R(T) $ which is the same as $ \mc R(T^*) $.  On the other hand, one may ask whether every closed subspace $\mc M $ of $\mc H $ is the range of some $ EP $ operator (not necessarily normal) on $\mc H $.  The answer is in the affirmative in a finite dimensional Hilbert space $\mc H$.    We give a procedure to construct such $ EP $ matrices and this construction has been used in the sequel to provide suitable examples of $ EP $ matrices.
We use the letters $ S, T $ for $ EP $ operators ; $ M, N $ for normal operators and $ A, B $ for bounded operators.

We start with some known characterizations of $ EP $ operators.
\begin{theorem}\cite{Pearl1966(2),Brock1990}
	Let $ T\in \mc B_c(\mc H) $.  Then the following are equivalent :
	\begin{enumerate}
		\item $ T $ is $ EP $ ;
		\item $ TT^\dag=T^\dag T $ ;
		\item $\mc N(T)^\bot=\mc R(T)$ ;
		\item $\mc N(T)=\mc N(T^*)$ ;
		\item $T^*=PT$, where $P$ is some bijective bounded operator on $\mc H$.
		
	\end{enumerate}
\end{theorem}

\begin{example} Let  $T:\ell_2 \rightarrow \ell_2$ be defined by $$ T(x_1,x_2,x_3,x_4,x_5,\ldots)=(x_1+x_2, 2x_1+x_2+x_3,-x_1-x_3,x_4,x_5,\ldots).$$  Then $ T^*(x_1,x_2,x_3,x_4,x_5,\ldots)=(x_1+2x_2-x_3, x_1+x_2, x_2-x_3,x_4,\ldots) $ and $ \mc N(T)=\mc N(T^*)=\{(x_1,-x_1,-x_1,0,0,\ldots) : x_1\in \mathbb{C}\} $. But $ TT^*\neq T^*T $. Since $ \mc N(T) $ is finite dimensional, $ \mc R(T) $ is closed. Hence $ T $  is an $EP$ operator but not normal.

\end{example}

\begin{theorem}\label{th2.1}
	If $ \mc W $ is a subspace of $ \mathbb{C}^n $, then there exists an $ EP $ matrix $ T $ of order $ n $ such that $ \mc R(T)=\mc W $.
\end{theorem}
\begin{proof}
	\sloppy If $\mc  W $ is a trivial subspace of $\mathbb{C}^{n} $, then it holds trivially.
	Without loss of generality, let $ \mc W $ be a subspace of $ \mathbb{C}^n $ with of  dimension $ n-1 $. Then $ \mc W $ can be expressed as $$\Big\{(x_1, x_2, \ldots, x_{i-1}, \sum_{k=1}^{n-1}a_k x_k, x_i, \ldots, x_{n-1}) \ : \ x_k \in \mathbb{C}, k=1, 2, \ldots, n-1\Big\}.$$  Let $$\left\{v_j=\big(
	x_{j1}, x_{j2}, \ldots, x_{j(i-1)}, \sum_{k=1}^{n-1}a_{k}x_{jk}, x_{ji}, \ldots, x_{j(n-1)}\big), j=1,2,\ldots,n-1\right\}$$ be a basis for $\mc  W $ which can be regarded as column vectors.
	
	Take $ T= \left[\begin{array}{cccccccc}
	v_1& v_2& \cdots& v_{i-1}& v^\prime &v_i& \cdots  & v_{n-1} 
	\end{array}\right]$ where $$ v^\prime =\left(\sum_{k=1}^{n-1} \overline{a_k} x_{k1},\sum_{k=1}^{n-1}\overline{a_k}x_{k2},\ldots,\sum_{j=1}^{n-1}\sum_{k=1}^{n-1}a_{j}\overline{a_k}x_{kj},\ldots,\sum_{k=1}^{n-1}\overline{a_k}x_{k(n-1)}\right).$$ Since the columns of $ T $ contain a basis of $\mc  W $, $ \mc R(T)=\mc W $. Now we need to show that $ T $ is $ EP $. But the selection of $ v^\prime  $ ensures that each row of $ T $ is in $\mc  W $. Hence $\mc R(T^*)=\mc W$. Therefore the result is true  when dimension of $ \mc W $ is $ n-1 $. 
	
	For the sake of completeness we also prove the result when the dimension of $ \mc W$ is  $ n-2$ and continuing the same technique to construct  $ EP $ matrices for lesser dimension of $ \mc W $.  Suppose that $\mc W $ is of  dimension  $ n-2 $.  Then $ \mc W $ can be expressed as 
	\begin{eqnarray*}
		\Big\{(x_1, x_2, \ldots, x_{i-1}, \sum_{k=1}^{n-2}a_k x_k, x_i, \ldots,x_{\ell-1},\sum_{k=1}^{n-2}b_k x_k, x_\ell,\ldots, x_{n-2}) \ : \ x_k \in \mathbb{C}, \\k=1, 2, \ldots, n-2\Big\}.
	\end{eqnarray*}
	Let 
	\begin{eqnarray*}
		\Big\{v_j=\big(
		x_{j1}, \ldots, x_{j(i-1)}, \sum_{k=1}^{n-2}a_{k}x_{jk}, x_{ji}, \ldots,x_{j(\ell-1)},  \sum_{k=1}^{n-2}b_{k}x_{jk},x_{j\ell},\ldots x_{j(n-2)}\big),  \\j=1,2,\ldots,n-2\Big\}
	\end{eqnarray*}
	be a basis for $\mc  W $ which can be regarded as column vectors.
	
	Take $ T= \left[\begin{array}{cccccccccccc}
	v_1& v_2& \cdots& v_{i-1}& v^\prime &v_i& \cdots & v_{\ell-1}& v^{\prime\prime} & v_\ell & \cdots & v_{n-2} 
	\end{array}\right]$ where 
	\begin{eqnarray*}
		v^\prime =\Bigg(\sum_{k=1}^{n-2} \overline{a_k} x_{k1},\sum_{k=1}^{n-2}\overline{a_k}x_{k2},\ldots,\sum_{j=1}^{n-2}\sum_{k=1}^{n-2}a_{j}\overline{a_k}x_{kj},\ldots,\sum_{j=1}^{n-2}\sum_{k=1}^{n-2}b_{j}\overline{a_k}x_{kj},\ldots,\\ \sum_{k=1}^{n-2}\overline{a_k}x_{k(n-2)}\Bigg)
	\end{eqnarray*}
	and
	\begin{eqnarray*}
		v^{\prime\prime}= \Bigg(\sum_{k=1}^{n-2} \overline{b_k} x_{k1},\sum_{k=1}^{n-2}\overline{b_k}x_{k2},\ldots,\sum_{j=1}^{n-2}\sum_{k=1}^{n-2}a_{j}\overline{b_k}x_{kj},\ldots,\sum_{j=1}^{n-2}\sum_{k=1}^{n-2}b_{j}\overline{b_k}x_{kj},\ldots,\\ \sum_{k=1}^{n-2}\overline{b_k}x_{k(n-2)}\Bigg). 
	\end{eqnarray*}
	As in the first case,  $ \mc R(T)=\mc R(T^*)=\mc W $. 
\end{proof}

\begin{remark}\label{rem1}
	If $ T $ is a complex  $ EP $ matrix of rank $ 1 $, then it must be normal, by the result (\cite{Campbell2009}, Theorem 1.3.3): If $ T $ is a complex matrix of rank $1$, then its Moore-Penrose inverse is of the form $ T^\dag = \frac{1}{\alpha}T^*$, where $ \alpha =\mbox{ trace } (T^*T).$
\end{remark}
\begin{remark}\label{thm_2_1}
	If $ T $ is a real  $ EP $ matrix of rank 1, then it must be a symmetric matrix. Indeed, as in Remark \ref{rem1}, T is a normal matrix. Hence by spectral theorem $ T=UDU^{*} $, for some unitary matrix $ U $  and  $$ D=\left[\begin{matrix}
	d & \mathbb O\\
	\mathbb O& \mathbb O
	\end{matrix}\right] $$
	where $ d= \mbox{ trace } (T) $ and $\mathbb O $ is the zero matrix of appropriate order. As $ T $ is real and $ T^*=UD^*U^* $, we have $ D=D^* $ and hence $ T $ is symmetric.
\end{remark}


\begin{example}
	Let $\mc  W=\{(x_1,x_1+x_2,x_2) \ : \  x_1,x_2\in \mathbb{C}\} $ be a subspace of $ \mathbb{C}^3 $ with basis $ v_1=(1,1+i,i), v_2= (1,0,-1)$. By the proof of the Theorem \ref{th2.1}, we have  $ v^\prime=(2,1+i,1) $. Then $ T=\left[\begin{array}{ccr}
	1&2&1\\
	1+i&1+i&0\\
	i&i-1&-1
	\end{array}\right] $. Here $ T $ is an $ EP $ matrix (non-normal) with $ \mc R(T)=\mc W $.
\end{example}

\begin{conjecture}
	Let $ \mc W $ be a closed subspace of a Hilbert space $ \mc H $.  Then there exists an $ EP $ (non-normal) operator $T$ on $ \mc H $ such that $ \mc R(T) =\mc W$.
\end{conjecture}
\section{Fuglede-Putnam type theorems for $ EP $ operators}
\noindent The well-known Fuglede theorem for a bounded operator is stated as follows.
\begin{theorem}\cite{Fuglede1950}\label{fug_thm}.
	Let $ N \in \mc B(\mc H) $ be a normal operator and $ A\in \mc B(\mc H)  $. If $ AN =NA $, then $ AN^* = N^*A $.
\end{theorem}

The following example illustrates that Fuglede theorem does not hold good for $ EP $ operators.  The theorem cannot be extended to the set of $ EP $ operators on $ \mc H $ even though every normal operator with a closed range is $ EP $.
\begin{example}\label{ex_1}
	\sloppy Consider the $ EP $ operator $ T $ on $ \ell_2 $ defined by $ T(x_1,x_2,x_3,\ldots)=(x_1-x_2,x_1+x_3,2x_1-x_2+x_3,x_4,\ldots) $ and $ A \in \mc B(\ell_2)$ defined by $ A(x_1,x_2,x_3,\ldots)=(x_2,-x_1+x_2-x_3,-2x_1+x_2,x_4,\ldots) $. Here $ AT=TA $ but $ AT^*\neq T^*A $.	
	
\end{example}
We have seen in the above example that Fuglede theorem is not true in general for  $ EP $ operators.  The following theorem is a Fuglede type theorem which proves that if an $ EP $ operator and a bounded operator commute, then the $ EP $  operator commutes with the Moore-Penrose inverse of the bounded operator. Our result just replaces the ``adjoint'' operation by the ``Moore-Penrose inverse'' in the Fuglede theorem stated in Theorem \ref{fug_thm}, however proofs are totally different.

\begin{theorem}\label{thm_3_1}
	Let $T$ be an $EP$ operator on $\mc H$ and $A\in \mc B(\mc H)$. If $AT=TA$, then $AT^\dag=T^\dag A$.
\end{theorem}
\begin{proof}
	As $T$ is an $EP$ operator, we have $TT^\dag =T^\dag T$. From the assumption $AT= TA$,  we have $AT^\dag=AT^\dag TT^\dag=AT(T^\dag)^2=TA(T^\dag)^2= TT^\dag TA(T^\dag)^2=T^\dag TAT(T^\dag)^2=T^\dag TAT^\dag=T^\dag ATT^\dag =T^\dag TT^\dag AT^\dag T=(T^\dag)^2TAT^\dag T=(T^\dag)^2ATT^\dag T=(T^\dag)^2AT=(T^\dag)^2TA=T^\dag A$.
\end{proof}

\begin{example}
	The assumption that $T$ is an $EP$ operator cannot be dropped in  Theorem \ref{thm_3_1}. For instance, $ A=T $ is an bounded operator on $ \ell_2 $ defined by $ T(x_1,x_2,x_3,\ldots)=(x_1+x_2,2x_1+2x_2,x_3,\ldots) $. Then $ T^\dag (x_1,x_2,x_3,\ldots)=(\frac{1}{10}(x_1+2x_2),\frac{1}{10}(x_1+2x_2),x_3,\ldots)  $. Note that $T$ is not an $EP$ operator and $ AT=TA$ but $AT^\dag\neq T^\dag A$.
\end{example}

Under some conditions, we prove that Fuglede theorem is true for  $ EP $ operators and we give examples which embellish that those conditions are necessary.

\begin{theorem}\label{thm_1_1}
	Let $T$ be an $EP$ operator on $\mc H$ and $A\in \mc B(\mc H)$. If $AT=TA$ and $AT^*T=T^*TA$, then $AT^*=T^*A$.
\end{theorem}
\begin{proof}
	Suppose $T\in \mc B(\mc H)$ is an $EP$ operator with $AT=TA$ and $AT^*T=T^*TA$. Then by Theorem \ref{thm_3_1},   we have $AT^*=A(TT^\dag T)^*=AT^*(TT^\dag)^*=AT^*TT^\dag=T^*TAT^\dag=T^*TT^\dag A=(TT^\dag T)^*A=T^*A$.
\end{proof}

\begin{theorem}\label{thm_1_2}
	Let $ T$ be an $ EP $ operator on $ \mc H $ and $ A\in \mc B(\mc H) $. If $ AT=TA $ and $ AT^\dag T^*=T^\dag T^*A $, then $ AT^*=T^*A $.
\end{theorem}
\begin{proof}
	As $ T\in \mc B(\mc H) $ is an  $ EP $ operator, we have $TT^\dag=T^\dag T$.  From the given facts $ AT=TA $ and $ AT^\dag T^*=T^\dag T^*A $, we have $ AT^*=A(TT^\dag T)^*=AT^\dag TT^*=ATT^\dag T^*=  T AT^\dag T^*=TT^\dag T^*A=(TT^\dag T)^*A=T^*A$.
\end{proof}

\begin{example}
	The condition $AT^*T=T^*TA$ is essential in Theorem \ref{thm_1_1}. Consider the $EP$ operator $T$ on $ \ell_2 $ defined by $ T(x_1,x_2,x_3,\ldots)=(x_1+x_3,0,x_3,\ldots) $ 
	and $ A\in \mc B(\ell_2) $ defined by $ A(x_1,x_2,x_3,\ldots)=(x_1+2x_3,-x_2,x_3,\ldots)$. Then $AT=TA $ and  $AT^*T \neq T^*TA $. But 
	$AT^*\neq T^*A $.

\end{example}

\begin{example}
	The condition $ AT^\dag T^*=T^\dag T^*A $ cannot be dropped in Theorem \ref{thm_1_2}. Let $A$ and $T$ be as in Example \ref{ex_1}. Then $ AT=TA $  and $AT^\dag T^*\neq T^\dag T^*A$. But $AT^*\neq T^*A$.
\end{example}

%

Fuglede theorem was generalized for two normal operators by Putnam, which is well-known as Fuglede-Putnam theorem and is stated as follows.
\begin{theorem}\cite{Putnam1951}
	Let $N,M$ be bounded normal operators on $\mc H$ and  $A\in \mc B(\mc H)$. If $AN = MA$, then $AN^* = M^*A$.
\end{theorem}
Fuglede-Putnam theorem is not true in general if we replace bounded normal operators by $EP$ operators, as shown in the following example.
\begin{example}\label{ex_2_14}Consider the $EP$ operators $ T $ and $ S $ on $ \ell_2  $ are defined by $$T(x_1,x_2,x_3,\ldots)=(x_1+x_3,0,x_3,\ldots)   $$ and $S(x_1,x_2,x_3,\ldots)=(x_1+x_2,x_2,,0,x_4,\ldots)   $ and $ A\in \mc B(\ell_2) $ is defined by $ A(x_1,x_2,x_3,\ldots)=(x_1-x_3,x_3,2x_2,x_4,\ldots) $.  Then $ AT=SA $. But $ AT^*\neq S^*A.$ 
\end{example}

\begin{theorem}\label{thm_1_2_1}
	Let $T,S$ be $EP$ operators on $ \mc H $ and $A\in \mc B(\mc H)$. If $AT=SA$ and $AT^*T=S^*SA$, then $AT^*=S^*A$.
\end{theorem}
\begin{proof}
	Suppose that $T,S\in \mc B(\mc H)$ are $EP$ operators with $AT=SA$ and $AT^*T=S^*SA$. Then we have $AT^*=A(TT^\dag T)^*=AT^*TT^\dag=S^*SAT^\dag=S^*SS^\dag A=(SS^\dag S)^*A=S^*A$.
\end{proof}
\begin{example}
	The condition  $AT^*T=S^*SA$ in Theorem \ref{thm_1_2_1} is essential. Let $ T,S$ be $ EP $ operators  and $ A $ be  the operator as in Example \ref{ex_2_14}. Here $ AT^*T\neq S^*SA $ and $ AT=SA $ but $ AT^*\neq S^*A$.
\end{example}
\begin{theorem}\label{thm_1_3}
	Let $ T,S$ be  $ EP $ operators on  $\mc H$  and $ A\in \mc B(\mc H) $. If $ AT=SA $ and $ AT^\dag T^*=S^\dag S^*A $, then $ AT^*=S^*A $.
\end{theorem}

\begin{proof}
	As $ T$ and $ S$ are $ EP $ operators with $ AT=SA $ and $ AT^\dag T^*=S^\dag S^*A $, we have $ AT^*=A(TT^\dag T)^*=AT^\dag TT^*= ATT^\dag T^*=SAT^\dag T^*=SS^\dag S^*A=(SS^\dag S)^*A=S^*A. $
\end{proof}

\begin{example}
	The condition $ AT^\dag T^*=S^\dag S^*A $ in Theorem \ref{thm_1_3} is essential. Let $ T,S$ be $ EP $ operators  and $ A $ be  the operator as in Example \ref{ex_2_14}.  Here 
	$ AT^\dag T^* \neq S^\dag S^* A$  and $ AT=SA $ but $ AT^*\neq S^*A$.
\end{example}

The following Fuglede-Putnam type theorem for $ EP $ operators is a generalization of Theorem \ref{thm_3_1} involving two $EP$ operators.
\begin{theorem}\label{thm_3_2}
	Let $ T,S$ be  $ EP $ operators on $ \mc H $ and  $A\in \mc B(\mc H)$. If $AT=SA$, then $AT^\dag=S^\dag A$.
\end{theorem}
\begin{proof}
	As $T$ and $S$ are $EP$ operators, we have $TT^\dag=T^\dag T$ and  $SS^\dag=S^\dag S$.  From the given fact $AT= SA$, we have $AT^\dag=AT^\dag TT^\dag=AT(T^\dag)^2=SA(T^\dag)^2= SS^\dag SA(T^\dag)^2~~=~~S^\dag SAT(T^\dag)^2~~=~~S^\dag SAT^\dag~~=~~~S^\dag ATT^\dag ~~~=\\S^\dag SS^\dag AT^\dag T=(S^\dag)^2SAT^\dag T=(S^\dag)^2ATT^\dag T=(S^\dag)^2AT=(S^\dag)^2SA=S^\dag A$.
\end{proof}
\begin{example}
	In the Theorem \ref{thm_3_2}, if one of the operators, $ T $ or $ S $ fails to be $ EP $, then the theorem is not valid. Consider the $EP$ operator $T$ on $ \ell_2 $ defined by $ T(x_1,x_2,x_3,\ldots)=(x_1+x_3,0,x_3,\ldots) $  and the non-$EP$ operator $ S $ on $ \ell_2 $ defined by $ S(x_1,x_2,x_3,\ldots)=(x_1+x_2,0,0,x_4,\ldots) $. Let $ A\in \mc B(\ell_2)$ be defined by  $ A(x_1,x_2,x_3,\ldots)=(x_2+2x_3,-x_2,-x_2,x_4) $. Then $ AT=SA $. But $ AT^\dag \neq S^\dag A$.
\end{example}
\begin{theorem}\label{them1}
	Let $ T,S$ be $ EP $ operators on $ \mc H $. If $ A,B\in \mc B(\mc H) $ with $ AT=SB $ and $ AT^2=S^2B $, then $ AT^\dag =S^\dag B $.
\end{theorem}
\begin{proof}
	Suppose that $ A,B,T,S\in \mc B(\mc H) $ with $ AT=SB $ and $ AT^2=S^2B $, where $ T$ and $S$ are $ EP $ operators. Then $ AT^\dag=A(T^\dag T T^\dag)=ATT^\dag T^\dag= SBT^\dag T^\dag=SS^\dag SBT^\dag T^\dag=S^\dag S^2BT^\dag T^\dag=S^\dag AT^2T^\dag T^\dag=S^\dag ATT^\dag =S^\dag S^\dag S ATT^\dag= S^\dag S^\dag S^2BT^\dag=S^\dag S^\dag A T^2T^\dag =S^\dag S^\dag AT=S^\dag S^\dag SB=S^\dag B$.
\end{proof}
\begin{example}
	The assumptions that $ T$ and $S$ are $ EP $ operators in Theorem \ref{them1} cannot be dropped. For instance, let $ A,B,T,S \in \mc B(\ell_2) $ be defined by $ A(x_1,x_2,x_3,\ldots)=(x_2,x_1,x_3,\ldots),\\ B=I, T(x_1,x_2,x_3,\ldots)=(x_1+x_2,-x_1-x_2,x_3,\ldots)$ and $ S(x_1,x_2,x_3,\ldots)=(-x_1-x_2,x_1+x_2,x_3,\ldots) $.  Here both $ T,S $ are not $ EP $ operators with $ AT=S=SB $. But $ AT^\dag\neq S^\dag B$.
\end{example}
\begin{example}
	The condition $ AT^2=S^2B $ in Theorem \ref{them1} is essential. For instance, let $ T,S\in \mc B(\ell_2) $  be $ EP $ operators defined by $T(x_1,x_2,x_3,\ldots)=(x_1-x_2,x_1+x_3,2x_1-x_2+x_3,x_4, \ldots) $ and $ S(x_1,x_2,x_3,\ldots)=(x_1+x_2,x_2,x_3,\ldots) $ and let $ A,B\in \mc B(\ell_2) $ be defined by $ A(x_1,x_2,x_3,\ldots)=(x_1+2x_2-x_3,-x_1-x_2+x_3,2x_1+2x_2-2x_3,x_4, \ldots) $ and $B(x_1,x_2,x_3,\ldots)=(x_1+x_3,0,x_1+x_2,x_4, \ldots)  $  be  such that $ AT=SB $ and $ AT^2 \neq S^2B $. But $ AT^\dag\neq S^\dag B$.
\end{example}

\begin{theorem}\label{thm_3_3}
	Let $T$ be an $EP$ operator on $ \mc H $ and $A, B\in \mc B(\mc H)$. If $AT=TB$ and $BT=TA$, then $AT^\dag=T^\dag B$ and $BT^\dag=T^\dag A$.
\end{theorem}
\begin{proof}
	From given hypotheses, $(A+B)T=T(A+B)$. By Theorem \ref{thm_3_1},
	\begin{eqnarray}\nonumber
		(A+B)T^\dag&=&T^\dag(A+B)\\ \nonumber
		AT^\dag+BT^\dag&=&T^\dag A+T^\dag B\\
		AT^\dag-T^\dag B&=&T^\dag A-BT^\dag. \label{eqn_3_1}
	\end{eqnarray}
	Again using given hypotheses, $(A-B)T=-T(A-B)$. By Theorem \ref{thm_3_2},
	\begin{eqnarray}\nonumber
		(A-B)T^\dag&=&-T^\dag(A-B)\\ \nonumber
		AT^\dag-BT^\dag&=&-T^\dag A+T^\dag B\\
		AT^\dag-T^\dag B&=&-T^\dag A+BT^\dag.\label{eqn_3_2}
	\end{eqnarray}
	
	Adding (\ref{eqn_3_1}) and (\ref{eqn_3_2}), we have  $AT^\dag=T^\dag B$. Similarly subtracting (\ref{eqn_3_2}) from (\ref{eqn_3_1}), we have $BT^\dag=T^\dag A$.
\end{proof}

\begin{theorem}
	Let $T, S$ be $EP$  operators on $ \mc H $ and $A, B\in \mc B(\mc H)$. If $AT=SB$ and $BT=SA$, then $AT^\dag=S^\dag B$ and $BT^\dag=S^\dag A$.
\end{theorem}
\begin{proof}
	From given hypotheses, $(A+B)T=S(A+B)$. By Theorem \ref{thm_3_2},
	\begin{eqnarray}\nonumber
		(A+B)T^\dag&=&S^\dag(A+B)\\ \nonumber
		AT^\dag+BT^\dag&=&S^\dag A+S^\dag B\\
		AT^\dag-S^\dag B&=&S^\dag A-BT^\dag.\label{eqn_3_1_1}
	\end{eqnarray}
	Again using given hypotheses, $(A-B)T=-S(A-B)$. By Theorem \ref{thm_3_2},
	\begin{eqnarray}\nonumber
		(A-B)T^\dag&=&-S^\dag(A-B)\\ \nonumber
		AT^\dag-BT^\dag&=&-S^\dag A+S^\dag B\\
		AT^\dag-S^\dag B&=&-S^\dag A+BT^\dag.\label{eqn_3_2_1}
	\end{eqnarray}
	
	Adding (\ref{eqn_3_1_1}) and (\ref{eqn_3_2_1}), we have  $AT^\dag=S^\dag B$. Similarly subtracting (\ref{eqn_3_2_1}) from (\ref{eqn_3_1_1}), we have $BT^\dag=S^\dag A$.
	
\end{proof}

\section{Consequences of Fuglede-Putnam type theorems for $ EP $ operators}

\noindent The product of $EP$ operators is not an $EP$ operator in general.
\begin{example}
	Let  $ S,T\in \mc B(\ell_2) $ be defined by $ S(x_1,x_2,x_3,\ldots)=(x_1+x_2,x_1+x_2,x_3,\ldots)  $ and   $  T(x_1,x_2,x_3,\ldots)=(0,x_2,x_3,\ldots)  $. Here $S$ and $T$ are $EP$ operators, 
	but the product  $ST$ is not an $EP$ operator.
\end{example}

Djordjevi{\'c}  has given a necessary and sufficient condition for product of two $EP$ operators  to be an $EP$ operator again.
\begin{theorem}\cite{Dragan2001}
	Let $S,T$ be $EP$ operators on $ \mc H $. Then the following statements are equivalent:
	\begin{enumerate}
		\item $ST$ is an $EP$ operator ;
		\item $\mc R(ST)=\mc R(S)\cap \mc R(T)$ and $\mc N(ST)=\mc N(S)+\mc N(T)$.
	\end{enumerate}
\end{theorem}

The following example illustrates the fact that there are operators $S$ and $T$ in $\mc B_c(\mc H)$ such that $ST\in \mc B_c(\mc H)$ but $TS\notin \mc B_c(\mc H)$. We have proved that when $S$ and $T$ are $EP$ operators, the closed rangeness of $ST$ implies the closed rangeness of $TS$ and vice-versa.
\begin{example}\cite{Sam2006}
	Let $S$ be an operator on $\ell_2$ defined by $S(x_1,x_2,x_3,\ldots)=(x_1,0,x_2,0,\ldots)$ and $T$ be another operator on $\ell_2$ defined by $T(x_1,x_2,x_3,\ldots)=(\frac{x_1}{1}+x_2,\frac{x_3}{3}+x_4,\frac{x_5}{5}+x_6,\ldots)$. One can verify that both $S$ and $T$ are bounded operators and are having closed ranges. Also, $\mc R(ST)$ is closed but $\mc R(TS)$ is not closed.
\end{example}

\begin{theorem}\cite{Sam2018}
	Let $S$ and $T$ be $EP$ operators on $\mc H$. Then $\mc R(ST)$ is closed if and only if $\mc R(TS)$ is closed.
\end{theorem}

\begin{example}
	Consider the $EP$ operators $ S,T\in \mc B(\ell_2) $ defined by $$S(x_1,x_2,x_3,\ldots)=(x_1+x_2,x_2,x_3,\ldots)  $$ and $T(x_1,x_2,x_3,\ldots)=(x_1,0,x_3,\ldots)  $. Here $ST$  is an $ EP $ operator, but $ TS $ is not $ EP $. 
\end{example}

\begin{theorem}\label{them_3_5}
	Let $S, T\in \mc B(\mc H)$ such that $(ST)^\dag=T^\dag S^\dag$.  Then $ST$ and $TS$ are $EP$ if and only if $S^\dag ST=TSS^\dag$ and  $STT^\dag=T^\dag TS$.
\end{theorem}
\begin{proof}
	Suppose $ST$ and $TS$ are $EP$. Then $(ST)^\dag$ and $(TS)^\dag$ are also $EP$. Hence we have $S^\dag (ST)^\dag=S^\dag T^\dag S^\dag=(TS)^\dag S^\dag$. Therefore by Theorem \ref{thm_3_2}, we have $S^\dag ST=TSS^\dag$. In a similar way we have $(ST)^\dag T^\dag=T^\dag S^\dag T^\dag=T^\dag (TS)^\dag$. Now we use Theorem \ref{thm_3_2}, we get $STT^\dag=T^\dag TS$. 
	Conversely, suppose  we have
	\begin{eqnarray}
		S^\dag ST&=&TSS^\dag\label{eqn_4_1}\\
		STT^\dag&=&T^\dag TS \label{eqn_4_2}.
	\end{eqnarray}
	From the equation (\ref{eqn_4_1}), we get $T^\dag S^\dag ST=T^\dag T SS^\dag$ and from the equation (\ref{eqn_4_2}), we get $STT^\dag S^\dag=T^\dag TSS^\dag$. Since the right side of these two equations are same, we have $T^\dag S^\dag ST= STT^\dag S^\dag$. Hence $(ST)^\dag S T=ST(ST)^\dag$.
	Therefore $ST$ is $EP$. Similarly from the equation (\ref{eqn_4_1}), we get $S^\dag STT^\dag=TSS^\dag T^\dag$ and from the equation (\ref{eqn_4_2}), we get $S^\dag STT^\dag=S^\dag T^\dag TS$. Therefore $TSS ^\dag T^\dag=S^\dag T^\dag TS$. Hence $TS(TS)^\dag=(TS)^\dag TS$. Thus $TS$ is $EP$.
\end{proof}

\begin{corollary}\label{thm_3_4}
	Let $S=UP\in \mathbb{C}^{n\times n}$ be a polar decomposition of $S$ where $U\in \mathbb{C}^{n\times n}$ is unitary and $P\in \mathbb{C}^{n\times n}$ is positive semidefinite Hermitian and let $T \in \mathbb{C}^{n\times n}$ with $(ST)^\dag=T^\dag S^\dag$. If $TU$ is $EP$ and $PTU=TUP$, then $ST$ and $TS$ are $EP$.
\end{corollary}
\begin{proof}
	Suppose $TU$ is $EP$ and $PTU=TUP$, then $TSS^\dag=T(UP)(UP)^\dag=TUPP^\dag U^*=PTUP^\dag U^*=PP^\dag TUU^*=PP^\dag T=P^\dag PT=P^\dag U^*UPT=(UP)^\dag UP T=S^\dag ST$. Since $TU$ is $EP$ and $PTU=TUP$, we have $P(TU)^\dag=(TU)^\dag P$. Therefore $STT^\dag=UPTUU^*T^\dag=UPTU(TU)^\dag=UTUP(TU)^\dag=UTU(TU)^\dag P=U(TU)^\dag TU P=$  $UU^*T^\dag TUP=T^\dag TS$. Thus by Theorem \ref{them_3_5}, $ ST $ and $ TS $ are $ EP $.
\end{proof}

\bibliographystyle{amsplain}
\bibliography{hypo_EP.bib}

\end{document}